\documentclass[12pt, a4, reqno]{amsart}

\usepackage{enumerate}
\usepackage{amsaddr}
\title{On signs of coefficients of $L$-functions}
\author{Didier Lesesvre}

\address{CNRS --- Université de Montréal CRM-CNRS \\Universit\'e de Lille -- Laboratoire Paul Painlev\'e, UMR 8524,
 59000 Lille, France\\
 e-mail: didier.lesesvre@univ-lille.fr}

\author{ Ming Ho Ng} 
 \address{Department of Mathematics, The Chinese University of Hong Kong\\ Shatin, Hong Kong, P.R. China\\e-mail: mhng@math.cuhk.edu.hk} 
 \author{Yingnan Wang}
 \address{School of Mathematical Sciences, Shenzhen University\\
Shenzhen, Guangdong 518060, P.R. China\\e-mail: ynwang@szu.edu.cn}
\date{\today}
\usepackage{fullpage}
\usepackage{amsmath, amssymb, amsthm}
\usepackage{xcolor}
\usepackage[colorlinks=true, linkcolor=blue]{hyperref}

\newcommand{\GSp}{\mathrm{GSp}}
\newcommand{\GL}{\mathrm{GL}}

\newcommand{\R}{\mathbb{R}}

\newtheorem{prop}{Proposition}
\newtheorem{thm}{Theorem}
\newtheorem{lem}{Lemma}

\newtheorem{coro}{Corollary}
\theoremstyle{remark}
\newtheorem{rk}{Remark}

\renewcommand{\leq}{\leqslant}
\renewcommand{\geq}{\geqslant}

\begin{document}

\begin{abstract}
We give a  general lower bound on the frequency of sign changes in the real coefficients of $L$-functions of the Selberg class. We in particular recover existing results in the cases of $\GL(2)$ and $\GL(3)$, and obtain new bounds in the case of $\GSp(4)$.
\end{abstract}

\maketitle

\section{Introduction}

$L$-functions are ubiquitous in number theory and carry critical information on objects of many natures: elliptic curves, Galois representations, automorphic forms, etc. see \cite{iwaniec_perspectives_2000} for a more in-depth survey. Whereas some results on $L$-functions depend on very specific settings where extra arithmetical or structural information is available (e.g. Hecke relations, trace formulas, functoriality, etc.), some arguments remain general and presenting them in such a generality helps understanding how properties are connected and motivates certain hypotheses. We address here, in this general setting, the question of establishing statistics on the sign changes on the coefficients attached to such $L$-functions, motivated by problems such as Linnik's problem in the case of Dirichlet $L$-functions and the recent work of Jääsaari \cite{jaasaari_signs_2022} in the case of $\GL(3)$.

\subsection{$L$-functions and the Selberg class} 
\label{subsec:$L$-functions}

We define in this section a general $L$-function in the Selberg class, referring to \cite[Chapter 5]{iwaniec_analytic_2004}. An $L$-function $L(s)$ is a complex analytic function on a right half-plane having the following properties: 
\begin{enumerate}[(i)]
\item \textbf{Dirichlet series:} It admits a series expansion of the form
\begin{equation}
\label{dirichlet-series}
L(s) = \sum_{m \geqslant 1} \frac{A(m)}{m^s}, \qquad \Re(s) > 1, 
\end{equation}
where $A(m) \in \mathbb{C}$ for all $m \geqslant 1$ and are called the coefficients of the $L$-function.
\item \textbf{Euler product:} We assume that $L(s)$ admits a Euler product 
\begin{equation}
\label{euler-product}
L(s) = \prod_p L_p(s), \qquad \Re(s) > 1, 
\end{equation}
where the product runs over primes $p$, and for every prime $p$ we let
\begin{equation}
L_p(s)  =\sum_{k\geqslant 0}\frac{A(p^k)}{p^{ks}}.
\end{equation}
\item \textbf{Functional equation:} There is an integer $d \geqslant 1$ (called the \textit{degree} of the $L$-function) and constants $\mu_j \in \mathbb{C}$ such that, letting 
\begin{equation}
    \gamma(s) = \prod_{j=1}^d \Gamma\left( \frac{s+\mu_j}{2}\right), 
\end{equation}
the $L$-function $L(s)$ extends meromorphically to the whole complex plane $\mathbb{C}$ and satisfies the functional equation
\begin{equation}
    L(s) = \varepsilon(L) \frac{\gamma(1-s)}{\gamma(s)} L(1-s)
\end{equation}
for a certain $\varepsilon(L) \in \mathbb{U}$.
\item \textbf{Subconvexity:}
We consider a (sub)convexity exponent $\theta \geqslant 0$ towards Lindelöf hypothesis, i.e. such that
\begin{equation}
\label{subconvexity-theta}
L(\tfrac12 + it) \ll (1+|t|)^{\theta + \varepsilon}, 
\end{equation}
which, by the Phragmén-Lindelöf principle, implies that in the critical strip $\sigma \in (0,1)$ we have 
\begin{equation}\label{convexity bound}
L(\sigma+it) \ll (1+|t|)^{2\theta(1-\sigma) + \varepsilon}.
\end{equation}
Unconditionally for any $L$-function satisfying mild  bounds on coefficients (viz. $A(m) \ll m^{1/2-\delta}$ for a certain $\delta >0$, which is typically the case for $\GL(n)$ $L$-functions) we know the convexity bound $\theta = d/4$ where the integer $d$ is the {degree} of the $L$-function. The Lindelöf hypothesis postulates that $\theta = 0$; and some explicit exponent beyond convexity are established in specific cases of low rank, and recently for all $\GL(n)$ by Nelson \cite{nelson_bounds_2023}.
\item \textbf{Rankin-Selberg bounds:}
We assume the Rankin-Selberg bound, i.e., for all $X \geqslant 1$, 
\begin{equation}\label{R-S bound}
\sum_{m \leq X} |A(m)|^2 \ll X^{1+\varepsilon}
\end{equation}
which corresponds to the non-existence of poles to the corresponding Rankin-Selberg $L$-function on $\Re(s)>1$, and of classical zero-free regions around $\Re(s) = 1$. As an application, we easily obtain that
\begin{equation}\label{bound for coefficients}
A(m) \ll m^{1/2+\varepsilon}.
\end{equation}
\item \textbf{Lower bound on non-vanishing of coefficients:}
We require lower bounds on sums of coefficients of the form
\begin{equation}
\label{lower-bound-coeff}
\sum_{m \sim X} |A(m)| \gg X^{\kappa  - \varepsilon}
\end{equation}
where $\kappa  > 0$. By \eqref{R-S bound} and Cauchy's inequality, we have $\kappa\leq 1$. If we apply the Rankin-Selberg theory and a bound toward the Generalized Ramanujan Conjecture $A(m)\ll m^{\vartheta} $, we have that $\kappa=1-\vartheta$ is admissible. If we assume the generalized Sato-Tate conjecture, we can take $\kappa=1-c \log \log (X)/\log (X)$ for $X$ sufficiently large, in particular $\kappa \geqslant 1-\varepsilon$ for all $\varepsilon >0$ when $X$ grows to infinity.
\end{enumerate}	

\subsection{Signs of coefficients}

We assume that the coefficients $A(m)$ are real, which is, for instance implied by a self-contragredience property of the underlying object. Assuming that $L(s)$ has no pole at $s=1$ (this excludes $\zeta(s)$, and it is the only case for $L$-functions that are primitive in the sense of Selberg), Landau's argument (see for instance \cite{pitale_sign_2008}) implies that the $A(m)$ changes sign infinitely often, it is therefore relevant to understand this behaviour in a finer way.

Some results have been explored in this direction in recent years, in the case of $\GL(2)$ (see \cite{matomaki}), but recently also $\GL(3)$ (see \cite{jaasaari_signs_2022}) or $\GL(n)$ (see \cite{lau}). We address here similar questions in the case of the Selberg class, and refine the arguments from \cite{jaasaari_signs_2022} to show their robustness in this very general setting.
The main result of this paper is the following.

\begin{thm}
\label{thm}
Consider an $L$-function with real coefficients $(A(m))_{m\geqslant 1}$  in the Selberg class as described above --- in particular a bound towards Lindelöf hypothesis with exponent $\theta$, a bound towards non-vanishing with exponent $\kappa $ and a Rankin-Selberg estimate. We assume that 
\begin{equation}
    \begin{cases}
        \kappa\geq1-\left({2(7\theta+3+\sqrt{(7\theta+3)^2+2})^2}\right)^{-1} & \text{if } \theta \geqslant 1/2, \\
        \kappa \geqslant -\frac{17}{2}+3\sqrt{10} &  \text{if } \theta \leqslant 1/2.
    \end{cases}
\end{equation}
Then the number of sign changes in $(A(m))_{m\leqslant X}$ is at least
\begin{equation}
    \begin{cases}
        X^{2\kappa - 2 + 1/2\theta-\delta- \frac{6\delta + 2\delta^2}{2\theta}} & \text{if $\theta \geqslant 1/2$}, \\
        X^{2\kappa - 1 - 6\delta} & \text{if $\theta \leqslant 1/2, $}
    \end{cases}
\end{equation}
for  $\delta = \sqrt{2-2\kappa+\varepsilon}$.
\end{thm}

\begin{rk}
    Theorem \ref{thm} in particular includes the $\GL(2)$ case treated in \cite{matomaki} and the $\GL(3)$ case treated in \cite {jaasaari_signs_2022}. It improves upon the $\GL(n)$ case from \cite{lau} provided there are good enough bounds on $\theta$ and $\kappa$.
\end{rk}

\begin{rk}
Conjecturally, $\kappa = 1$, in which case the number of sign changes in $(A(m))_{m\leqslant X}$ is
\begin{equation}
    \begin{cases}
        X^{1/2\theta-\delta} & \text{if $\theta \geqslant 1/2$}, \\
        X^{1- \delta} & \text{if $\theta \leqslant 1/2, $}
    \end{cases}
\end{equation}
for all $\delta > 0$. Moreover, Lindelöf hypothesis postulates that $\theta = 0$, in which case there should always be $X^{1- \delta}$ sign changes for all $\delta >0$ (but Theorem \ref{thm} proves this is the case as soon as $\theta \leqslant 1/2$).
The conditions tying $\theta$ and $\kappa$ allow for a certain leeway even without having the best expected results towards the Lindelöf hypothesis or the non-vanishing of coefficients. If $\theta$ is very large, $\kappa$ is required to be very close to $1$; the convexity bound $\theta  = d/4$ already allows for $\kappa \geqslant 1-\eta$ for a certain $\eta>0$. When $\theta = 1/2$, any $\kappa \geqslant 0.9971$ is admissible. When $\theta < 1/2$, any $\kappa > 0.9869$ is admissible. 
\end{rk}

\begin{rk}
    Besides refining the arguments in \cite{jaasaari_signs_2022} to adapt them to the Selberg class, we also give detailed proofs of the (sub)convexity bound for partial sums (see Lemma \ref{lem:subconvexity-partial-sums}) and the removal of congruences (see Lemma \ref{lem:cs}) in general settings. These may be interesting to some readers.
\end{rk}
We can deduce from the general statement in Theorem \ref{thm} a new result for the $\mathrm{GSp}(4)$ case:

\begin{coro}
\label{coro}
Let $f$ be a Siegel modular form of weight $k$. Assume that $f$ is a Hecke eigenform, which is not a Saito–Kurokawa lift. Let $(A(m))_m$ be the coefficients of the associated spinor $L$-function $L(s, f, \mathrm{Spin})$, and let $\theta$ be the best exponent towards the Lindelöf hypothesis. Then, the number of sign changes in $(A(m))_{m\leqslant X}$ is $\gg X^{1/2\theta + \varepsilon}$.
\end{coro}

\begin{rk}
    
The case of Saito–Kurokawa lifts is well-understood: they correspond to those modular forms such that $A(m)>0$ for all $m \geqslant 1$, as proven by Pitale-Schmidt \cite{pitale_weissauer}; there is therefore no sign change in this case.
\end{rk}

The strategy of the proof is to detect sign changes in short intervals, by comparing 
\begin{equation}
S_1(X) := \Bigg| \sum_{\substack{mk \in [x, x+H] \\ m \sim M \\ (m,k)=1}} A(m) \Bigg| \qquad \text{and} \qquad S_2(X) := \sum_{\substack{mk \in [x, x+H] \\ m \sim M \\ (m,k)=1}} |A(m)| 
\end{equation}
for a certain $H>0$ and to prove that $S_1(X) < S_2(X)$ for large enough $X$, therefore detecting a sign change in the interval $[X, X+H]$. This would ensure at least $X/H$ sign changes in $[1,X]$; in particular, the smaller $H$, the better the result. We will prove the inequality by establishing an upper bound for (the second moment of) $S_1(X)$ (in Section \ref{sec:upper-bound}) and a lower bound for $S_2(X)$ (in Section \ref{sec:lower-bound}), to conclude the argument in Section \ref{sec:conclusion}. A novel application is given in Section \ref{sec:application}.

%
%
%
%

\section{Upper bound}
\label{sec:upper-bound}

We aim at proving the following second moment upper bound, following the strategy in~\cite{jaasaari_signs_2022}. 
\begin{prop}
Let $X^{6\delta} \ll H \ll X^{1-6\delta}$ and $M=X^\delta$, where $0< \delta <1$ is a positive constant. We have
\begin{equation}
\label{second-moment}
\int_X^{2X} \Bigg| \sum_{\substack{mk \in [x, x+H] \\ m \sim M \\ (m,k)=1}} A(mk) \Bigg|^2 dx \ll H^2X^{1-\delta^2}+X^{2\theta(1+\delta)+6\delta}H^{2-2\theta}.
\end{equation}
\end{prop}

The whole section is dedicated to its proof. 

\subsection{Preliminary lemmas}

 Introduce, for $\Re(s)>1$, 
 \begin{align*}
 M(s) & = \sum_{m \sim M} \frac{A(m)}{m^s}, \\
 K(s) & = \sum_{X/3M \leqslant k \leqslant 3X/M} \frac{A(k)}{k^s}.
 \end{align*}
By \eqref{R-S bound} and using the Cauchy-Schwarz inequality, we have that for $\Re s=\sigma\geq1/2$,
\begin{equation}\label{bound for M(s)}
    M(s)\ll \left( \sum_{m \sim M} \frac{|A(m)|^2}{m}\right)^{1/2}\left( \sum_{m \sim M} {m^{1-2\sigma}}\right)^{1/2}\ll M^{1-\sigma+\varepsilon}.
\end{equation}

Analogously to \eqref{bound for M(s)}, we have $K(s)\ll (X/M)^{1-\sigma+\varepsilon}$. A key point is to get a finer "subconvexity" bound for the partial sum $K(s)$. 

\begin{lem}[Subconvexity bound for partial sums]
\label{lem:subconvexity-partial-sums} 
For $s=1/2+it$ with $|t|\leq T$, we have
$$
K(s) \ll T^{\theta+\varepsilon} +(X/MT)^{1/2+\varepsilon},
$$
where $\theta$ is a subconvexity exponent as defined in \eqref{convexity bound}.
\end{lem}
\begin{proof}
By the truncated Perron's formula in \cite[Theorem II 2.3]{tenenbaum}, we have
\begin{align*}
 \sum_{X/3M \leqslant k \leqslant 3X/M} \frac{A(k)}{k^{s}}
    & = \frac{1}{2i\pi} \int^{\frac12+\varepsilon+ iT}_{\frac12+\varepsilon - iT} L(s+z) \frac{(3X/M)^z - (X/3M)^z}{z}dz\\
    & \qquad +O\left((X/M)^{1/2+\varepsilon}\sum_{k\geq1}\frac{|A(k)|}{k^{1+\varepsilon}(1+T|\log (X/Mk)|)}\right).
\end{align*}
By dividing the sum in the error term into two parts, depending on the range of $k$, we get that the error term is bounded by
\begin{align*}
    &\frac1T\sum_{|k-X/M|\geq X/2M}\frac{|A(k)|}{k^{1+\varepsilon}}+\sum_{|k-X/M|< X/2M}\frac{|A(k)|}{k^{1+\varepsilon}(1+TM(k-X/M)/X)}\\
    &\ll_\varepsilon \frac1T+\frac XM\sum_{|k-X/M|< X/2M\atop k\neq X/M}\frac{|A(k)|}{k^{1+\varepsilon}(X/M+T|k-X/M|)}+\frac{A(X/M)}{(X/M)^{1+\varepsilon}}.
\end{align*}
Applying Cauchy-Schwarz inequality as well as the bounds \eqref{R-S bound} and \eqref{bound for coefficients}, the above formula yields 
\begin{align*}
   & \ll \frac1T+\frac XM\left(\sum_{|k-X/M|< X/2M\atop k\neq X/M}\frac{|A(k)|^2}{k^{1+\varepsilon}}\right)^{1/2}\left(\sum_{|k-X/M|< X/2M\atop k\neq X/M}\frac{1}{k^{1+\varepsilon}(X/M+T|k-X/M|)^2}\right)^{1/2}+\frac{1}{(X/M)^{1/2+\varepsilon}}\\
    &\ll\frac1T+\frac{(X/M)^{1/2-\varepsilon}}{(T(T+X/M))^{1/2}}+\frac{1}{(X/M)^{1/2+\varepsilon}}\\
    &\ll \frac1{T^{1/2}}+\frac{1}{(X/M)^{1/2+\varepsilon}}.
\end{align*}
Therefore, we have
\begin{align*}
 \sum_{X/3M \leqslant k \leqslant 3X/M} \frac{A(k)}{k^{s}}
    &=\frac{1}{2i\pi} \int^{\frac12+\varepsilon+ iT}_{\frac12+\varepsilon - iT} L(s+z) \frac{(3X/M)^z - (X/3M)^z}{z}dz\\
    &\qquad +O\left((X/M)^{\varepsilon}\left(1+\frac{(X/M)^{1/2}}{T^{1/2}}\right)\right).
\end{align*}
Shifting the line of integration from the line $\Re z=1/2+\varepsilon$ to the line
$\Re z=\varepsilon$ and applying the convexity bound
$L(s+z) \ll (1+|t|)^{2\theta(1/2-\sigma)+\varepsilon}$, the above integral is bounded by
\begin{align*}
    &X^\varepsilon\int_{\varepsilon}^{1/2+\varepsilon}\frac{(X/M)^\alpha(1+|t|+T)^{2\theta(1/2-\alpha)+\varepsilon}}{T}d\alpha+X^\varepsilon\int_{-T}^T\frac{(1+|t+h|)^{\theta+\varepsilon}}{\varepsilon+|h|}dh\\
    &\ll X^\varepsilon  T^\varepsilon \frac{T^\theta+(X/M)^{1/2}}{T}+X^\varepsilon  T^{\theta+\varepsilon}.
\end{align*}
Therefore, we have that for $|t|\leq T$, 
\begin{align*}
    \sum_{X/3M \leqslant k \leqslant 3X/M} \frac{A(k)}{k^s}
    \ll T^{\theta+\varepsilon} +(X/MT)^{1/2+\varepsilon}, 
\end{align*}
finishing the proof.
\end{proof}

 We can remove congruence conditions modulo $d$  in partial sums and relate to the corresponding full sums at the cost of an explicit polynomial in $p^{-s}$ for $p\mid d$. This is stated in the following lemma. 
 \begin{lem}[Removing congruences]
 \label{lem:d}
For all squarefree $d$ and for all $\Re(s)>1$, we have 
 \begin{equation}
 \sum_{\substack{m\geq 1\\ m \equiv 0 (d)}} \frac{A(m)}{m^s} = d^{-s} \prod_{p \mid d} p^s (1-P(p^{-s})) \sum_{m\geq1}\frac{A(m)}{m^s}.
 \end{equation}
 \end{lem}
 
 \begin{proof}

We compare Euler products and use the squarefreeness of $d$. Indeed, we have 
\begin{align}
\sum_{\substack{m \equiv 0(d)}} \frac{A(m)}{m^s} & = \sum_{m} \frac{A(dm)}{(dm)^s} = \prod_{p \mid d} p^{-s} \sum_{k \geqslant 0} \frac{A(p^{k+1})}{p^{ks}}\prod_{p \nmid d} \sum_{k \geqslant 0} \frac{A(p^k)}{p^{ks}} 
\end{align}
and,  moreover, 
\begin{equation}
\sum_{k \geqslant 0} \frac{A(p^{k+1})}{p^{ks}} = p^{s}\sum_{k \geqslant 0} \frac{A(p^{k+1})}{p^{(k+1)s}}  = p^s (P(p^{-s})^{-1} - 1).
\end{equation}
Dividing by the Euler product of the full sum, which amounts to multiplying by $P(p^{-s})$ at each place $p$, gives the result.
 \end{proof}


\begin{lem}
\label{lem:cs}
We have, for $\Re(s) = 1/2$, 
\begin{equation*}
    \left| \sum_{d \leq M} \mu(d) \sum_{\substack{m \sim M \\ m \equiv 0(d)}} \frac{A(m)}{m^s} \sum_{\substack{X/3M \leqslant k \leqslant 3X/M \\ k \equiv 0(d)}} \frac{A(k)}{k^s}\right|^2 \ll X^\varepsilon \left| \sum_{m \sim M/\ell_1} \frac{A(m)}{m^s} \right|^2\left| \sum_{X/3M\ell_2 \leqslant k \leqslant 3X/M\ell_2} \frac{A(k)}{k^s} \right|^2, 
\end{equation*}
for certain integers $\ell_1, \ell_2 \geqslant 1$.
\end{lem}

\begin{proof}
Due to the presence of $\mu(d)$ inside the summand, all the $d$'s occurring in this proof are necessarily squarefree, even though we do not indicate it explicitly in the indexation of sums.
By Cauchy-Schwarz inequality, the left-hand side is bounded by
$$
\left(\sum_{d \leq M}\left|\sum_{\substack{m \sim M \\ m \equiv 0(d)}} \frac{A(m)}{m^s}\right|^2\right)\left(\sum_{d \leq M}\left| \sum_{\substack{X/3M \leqslant k \leqslant 3X/M \\ k \equiv 0(d)}} \frac{A(k)}{k^s}\right|^2\right) .
$$
This leads us to consider (recall that $\Re s=1/2$) the sums of the form
\begin{align*}
    \sum_{d \leq A}\left|\sum_{\substack{m \sim B \\ m \equiv 0(d)}} \frac{A(m)}{m^s}\right|^2
    &= \sum_{d \leq A}\left|\sum_{\ell=d\atop P(\ell)=d}^{2B}\frac{A(\ell)}{\ell^s}\sum_{\substack{m \sim B/\ell \\ (m,d)=1}} \frac{A(m)}{m^s}\right|^2\\
    &\leq \max_{d\leq A,\ell\leq 2B,P(\ell)=d}\left\{\left|\sum_{\substack{m \sim B/\ell \\ (m,d)=1}} \frac{A(m)}{m^s}\right|^2\right\} \sum_{d \leq A}\left|\sum_{\ell=d\atop P(\ell)=d}^{2B}\frac{|A(\ell)|}{\ell^{1/2}}\right|^2\\
     &\leq \max_{d\leq A,\ell\leq 2B,P(\ell)=d}\left\{\left|\sum_{\substack{m \sim B/\ell \\ (m,d)=1}} \frac{A(m)}{m^s}\right|^2\right\} \sum_{d \leq A}\left(\sum_{\ell=d\atop P(\ell)=d}^{2B}\frac{|A(\ell)|^2}{\ell}\right)\left(\sum_{\ell=d\atop P(\ell)=d}^{2B}1\right),
\end{align*}
where $A\ll B$ and $$P(\ell)=\prod_{p|\ell}p.$$
Note that
\begin{align}
    \sum_{\ell=d\atop P(\ell)=d}^{2B}1\ll \sum_{n_1+n_2+\cdots+n_{\omega(d)}\leq [\log (2B)]\atop n_1,n_2,\ldots,n_{\omega(d)}\geq 1}1\ll \binom{[\log (2B)]}{\omega(d)}\ll \left(\frac{c\log B}{\omega(d)}\right)^{\omega(d)}\ll e^{C\frac{(\log B)(\log\log\log B)}{\log\log B}},
\end{align}
by dividing $\omega(d)$ into two ranges $\omega(d)\ll (\log B)/(\log\log B)^2$ and $(\log B)/(\log\log B)^2\ll \omega(d)\ll (\log B)/(\log\log B)$. Here $c,C$ are two absolute constants and we have applied the well-known result
\begin{equation}\label{bound for omega(d)}
    \omega(d)\ll (\log d)/(\log\log d)\ll (\log A)/(\log\log A)\ll (\log B)/(\log\log B).
\end{equation}

Next, we have
\begin{align*}
 \sum_{d \leq A}\sum_{\ell=d\atop P(\ell)=d}^{2B}\frac{|A(\ell)|^2}{\ell}
 \leq   \sum_{d \leq A}\sum_{\substack{\ell \leq 2B \\ \ell \equiv 0(d)}} \frac{|A(\ell)|^2}{\ell}\leq B^\varepsilon\sum_{d \leq A}\sum_{\substack{\ell \leq 2B \\ \ell \equiv 0(d)}} \frac{|A(\ell)|^2}{\ell^{1+\varepsilon}}\leq B^\varepsilon\sum_{d \leq A}\sum_{\substack{\ell =1 \\ \ell \equiv 0(d)}}^\infty \frac{|A(\ell)|^2}{\ell^{1+\varepsilon}}.
\end{align*}
Hence, we apply Lemma \ref{lem:d} to the inner sum on the right hand side and obtain
\begin{align*}
\sum_{d \leq A}\sum_{\ell=d\atop P(\ell)=d}^{B}\frac{|A(\ell)|^2}{\ell}
&\ll B^{\varepsilon}\sum_{d \leq A}\prod_{p|d}\left(\frac{|A(p)|^2}{p^{1+\varepsilon}}-\frac{|A(p)|^2}{p^{2+2\varepsilon}}+\frac{1}{p^{3+3\varepsilon}}\right)\sum_{\ell=1}\frac{|A(\ell)|^2}{\ell^{1+\varepsilon}}\\
&\ll B^\varepsilon\sum_{d \leq A}2^{\omega(d)}\prod_{p|d}\frac{|A(p)|^2+1}{p^{1+\varepsilon}}\sum_{\ell=1}\frac{|A(\ell)|^2}{\ell^{1+\varepsilon}}.
\end{align*}
Applying the Rankin-Selberg bound \eqref{R-S bound} and \eqref{bound for omega(d)}, we get 
$$
\sum_{d \leq A}\sum_{\ell=d\atop P(\ell)=d}^{B}\frac{|A(\ell)|^2}{\ell}\ll B^\varepsilon\sum_{d \leq A}\frac{\sum\limits_{h|d}|A(h)|^2}{d^{1+\varepsilon}}\ll B^\varepsilon\sum_{h \leq A}\frac{|A(h)|^2}{h^{1+\varepsilon}}\sum_{j\leq A/h}\frac1{j^{1+\varepsilon}}\ll B^\varepsilon.
$$
Moreover, the maximum shown above is reached for one of the finite possibilities of $\ell$.
Therefore, the original sum is bounded by
\begin{align*}
X^\varepsilon \left| \sum_{m \sim M/\ell_1} \frac{A(m)}{m^s} \right|^2\left| \sum_{X/3M\ell_2 \leqslant k \leqslant 3X/M\ell_2} \frac{A(k)}{k^s} \right|^2
\end{align*}
for some integers $\ell_1,\ell_2 \geqslant 1$, as claimed. 
\end{proof}

We will need the following mean-value theorem, see \cite[Chapter 9]{iwaniec_analytic_2004}, for Dirichlet polynomials:
\begin{lem}[Mean-value theorem]
\label{lem:mean-value-theorem}
Let $N \geqslant 1$ and let $F(s) = \sum_{n\sim N} a_n n^{-s}$ be a Dirichlet polynomial. Then
\begin{equation}
\label{mean-value-theorem}
\int_{-T}^T \left| F(\tfrac12 + it) \right|^2 dt \ll (N+T) \sum_{n \sim N}  \frac{|a_n|^2}{n}.
\end{equation}
\end{lem}

\subsection{Reducing to an integral via Perron formula}

We are now ready to study the sough second moment.
Perron's formula relates the sum of coefficients appearing therein to the inverse Mellin transform of the associated $L$-function: 
\begin{equation}
\sum_{\substack{mk \in [x, x+H] \\ m \sim M  \\ (m,k)=1}} A(mk) = \frac{1}{2i\pi} \int_{(2)} \sum_{\substack{m \sim M \\ X/3M \leqslant k \leqslant 3X/M\\(m,k)=1}} \frac{A(mk)}{(mk)^s} \frac{(x+H)^s - x^s}{s}ds.
\end{equation}

\subsection{Möbius to detect coprimality}

We appeal to the Saffari-Vaughan argument as in \cite{jaasaari_signs_2022}, multiplicativity of coefficients and Möbius inversion to detect the coprimality condition, so that the second moment \eqref{second-moment} is bounded by
\begin{equation}
\label{expanded-second-moment}
\int_X^{2X} \Bigg|\int_{(2)}  \sum_{d \leq M} \mu(d) \sum_{\substack{m \sim M \\ m \equiv 0(d)}} \frac{A(m)}{m^s} \sum_{\substack{ X/3M \leqslant k \leqslant 3X/M \\ k \equiv 0(d)}} \frac{A(k)}{k^s} x^s \frac{(1+u)^s - 1}{s}ds \Bigg|^2 dx, 
 \end{equation}
for some given $u \ll H/X$.  We will appeal various times to the bound
\begin{equation}
\label{fraction-bound}
\frac{(1+u)^s -1}{s} \ll \max\left( \frac{H}{X}, \frac{1}{|\Im(s)|}\right), 
\end{equation}
valid for all $u \ll H/X$.

\subsection{Moving the contours}

We now need to study the integral \eqref{expanded-second-moment}, by moving the contours. The integral in on a full vertical line inside the domain of absolute convergence of all the Dirichlet series therein, and we start by removing the high parts before moving contours. For simplicity, put
$$
F(s)=\sum_{d \leq M} \mu(d)\sum_{\substack{m \sim M \\ m \equiv 0(d)}} \frac{A(m)}{m^s} \sum_{\substack{k \sim X/M \\ k \equiv 0(d)}} \frac{A(k)}{k^s}.
$$

\subsubsection{High parts on the right}

We follow \textit{mutatis mutandis} Jääsaari's strategy in \cite{jaasaari_signs_2022}, and only recall briefly his argument for completeness. We bound the contribution of the large imaginary parts to the integral as follows: 
\begin{lem}
We have, for all $U>X^5$, for all $0<\Xi<1$, 
\begin{equation}
\int_X^{2X} \Bigg|  \int_{\substack{(1+\Xi) \\ \Im > U}} F(s)x^s \frac{(1+u)^s - 1}{s}ds \Bigg|^2 dx \ll X^{3+2\Xi}U^{-1}.
\end{equation}

\end{lem}

We open the square, which selects essentially the diagonal term (killing the oscillations $x^{s_1 + \bar{s}_2}$, which does matter since $x$ is large).  Indeed,  after expanding the square, we get
\begin{align}
& \iint_{\substack{(1+\Xi) \\ \Im > U}} \left|F(s_1)\overline{F(s_2)}\frac{(1+u)^{s_1}-1}{s_1}\overline{\frac{(1+u)^{s_2}-1}{s_2}}\right|  \times \left| \int_\R g(x/X) x^{s_1+\bar{s}_2} dx \right|ds_1ds_2, 
\end{align}
where $g$ is a compactly supported smooth function.

Now the integral inside is vanishingly small (i.e. $\ll X^{-A}$ for all $A>0$) if $\Im(s_1+\bar{s}_2) \neq 0$ by stationary phase. When $s_1 = s_2$, it is bounded by $X^{2\Re(s)+1}$. 
We then use the absolute convergence of $F(s)$ on $\Re(s)>1$, as well as the explicit bound on the $u$-term given in \eqref{fraction-bound}. We obtain a final bound on these terms of size $X^{3+2\Xi}/U$ and this is negligible when  $U>X^5$. 

\subsubsection{Horizontal parts}

By \eqref{bound for coefficients}, we obtain that for $\Re(s)=\sigma\geq1/2$,
\begin{equation*}\label{bound for F(s)}
    F(s)\ll X^{3/2-\sigma+\varepsilon}M.
\end{equation*}
Jääsaari's computations \cite{jaasaari_signs_2022} carry on until the bound (for $x \asymp X$)
\begin{equation}
\ll X\int_{1/2}^{1+\Xi} \left|F(\sigma\pm iU) x^\sigma \frac{(1+u)^{\sigma \pm iU} - 1}{\sigma \pm iU}\right|^2 d\sigma.
\end{equation}

%
%

Using that the fraction is $\ll 1/|t| \asymp 1/U$ by \eqref{fraction-bound} and squaring these bounds, we obtain altogether the bound
\begin{equation}
\ll X^{1+\varepsilon}\int_{1/2}^{1+\Xi} X^{3-2\sigma}M^2 U^{-2} d\sigma \ll X^4U^{-2}.
\end{equation}
This is negligible when $U>X^5$. 
%
%

\subsubsection{Bulk on the left}

These contours estimates explain that the integral we are studying is, up to the error terms displayed above, the same as the one on the vertical segment $[1/2\pm iU]$, that we shall study now. The second moment is therefore
\begin{equation}
\label{splitting}
\underbrace{H^2 \int_{1/2}^{1/2+iX/H} |F(s)|^2 ds}_{=: I_1} + \underbrace{X^2 \int_{1/2+iX/H}^{1/2+iU} \frac{|F(s)|^2}{|t|^2}dt}_{=: I_2}
\end{equation}
where the splitting has been sparked by the inequality \eqref{fraction-bound} that we used accordingly in each range. We will estimate separately each of these integrals.

We appeal to Lemma \ref{lem:cs} to get that

\begin{align}
I_1 \ll H^2 X^\varepsilon \int_{1/2}^{1/2+iX/H} \left|\sum_{\substack{m \sim M/\ell_1 }} \frac{A(m)}{m^s} \right|^2\left|\sum_{\substack{X/3M\ell_2 \leqslant k \leqslant 3X/M\ell_2 }} \frac{A(k)}{k^s}\right|^2 ds
\end{align}
and that 
\begin{align}
I_2 \ll X^{2+\varepsilon} \int_{1/2+iX/H}^{1/2+iU} \frac{1}{|t|^2} \left|\sum_{\substack{m \sim M/\ell_1 }} \frac{A(m)}{m^s} \right|^2\left|\sum_{\substack{X/3M\ell_2 \leqslant k \leqslant 3X/M\ell_2 }} \frac{A(k)}{k^s}\right|^2 ds
\end{align}
for some integers $\ell_1, \ell_2 \geqslant 1$. Note that since, ultimately, the bounds will be given by positive powers of $M$, we can ignore the $\ell_i$ and take them equal to $1$.

For the first integral in \eqref{splitting}, we split it into the set $\mathcal{T} = \{t \in [0, X/H] \ : \ |M(\tfrac12+it)| \leqslant M^{1/2-\delta}\}$ and its complementary $\mathcal{U}$ in $[0, X/H]$. By Lemma \ref{lem:mean-value-theorem} and \eqref{R-S bound}, the contribution of $\mathcal{T}$ is bounded by
\begin{align}\label{contribution of T}
&\int_\mathcal{T}\left|M\left(\frac12+it\right)\right|^2\left|K\left(\frac12+it\right)\right|^2dt\nonumber\\
&\ll M^{1-2\delta}\left(\frac XH+\frac XM\right)\sum_{ k \sim X/M}\frac{|A(k)|^2}{k}\ll X^{1+\varepsilon}M^{-2\delta}\ll X^{1-\delta^2}    
\end{align}
provided that $\varepsilon<\delta^2/10^{2026}$.

Since $M=X^\delta$ and $X^{2\delta}\ll H\ll X$, there exists a positive integer $\ell= [1/\delta]+1$ such that $M^\ell\gg X/H$.
By Cauchy-Schwarz inequatlity, Lemma \ref{lem:mean-value-theorem} and \eqref{R-S bound}, we have
\begin{align*}
    \mathrm{meas}(\mathcal{U})M^{2\ell(1/2-\delta)}
    &\leq\int_0^{X/H}\left|M\left(\frac12+it\right)\right|^{2\ell}dt\\
    &\ll\int_0^{X/H}\ell\sum_{j=0}^{\ell-1}\left|\sum_{n\sim 2^jM^\ell}\frac{\sum\limits_{n=m_1\cdots m_\ell\atop m_i\sim M}A(m_1)\cdots A(m_\ell)}{n^s}\right|^2dt\\
    &\ll \ell\sum_{j=0}^{\ell-1}\left(\frac XH+2^j M^\ell\right)\sum_{n\sim 2^jM^\ell}\frac{\left|\sum\limits_{n=m_1\cdots m_\ell\atop m_i\sim M}A(m_1)\cdots A(m_\ell)\right|^2}{n}\\
    &\ll \ell\sum_{j=0}^{\ell-1}\sum_{n\sim 2^jM^\ell}d_\ell(n)2^j
    M^\ell\sum\limits_{n=m_1\cdots m_\ell\atop m_i\sim M}\frac{\left|A(m_1)\cdots A(m_\ell)\right|^2}{n}\\
    &\ll \ell 2^\ell X^\varepsilon \left(\sum_{m\sim M}|A(m)|^2\right)^\ell\\
    &\ll \ell 2^{\ell} X^\varepsilon M^{\ell+\ell\varepsilon}.
\end{align*}
Hence, for $\varepsilon<\delta^2/10^{2026}$,
$$
\mathrm{meas}(\mathcal{U})\ll \ell 2^\ell X^{3\varepsilon} M^{2\ell\delta}\ll_\delta X^{4\delta}
$$
and by \eqref{bound for M(s)}, the contribution of $\mathcal{U}$ to the integral in $I_1$ is bounded by 
$$
\int_{\mathcal{U}}\left|M\left(\frac12+it\right)\right|^2\left|K\left(\frac12+it\right)\right|^2dt\ll X^{5\delta+\varepsilon}\sup_{|t|\leq X/H}\left|K\left(\frac12+it\right)\right|^2.
$$
Combining the above formula with \eqref{contribution of T} and Lemma \ref{lem:subconvexity-partial-sums}, we have
\begin{align}\label{I1}
I_1&\ll H^2X^{1-\delta^2}+H^2X^{5\delta+\varepsilon}\sup_{|t|\leq X/H}\left|K\left(\frac12+it\right)\right|^2\nonumber\\
&\ll H^2X^{1-\delta^2}+H^2X^{6\delta}\left((X/H)^{2\theta}+H/M\right).
\end{align}

Next, we turn to $I_2$ and it is bounded by
\begin{align*}
&X^{2+\varepsilon}\int_{X/H}^{MX/H}\frac1{|t|^2}\left|M\left(\frac12+it\right)\right|^2\left|K\left(\frac12+it\right)\right|^2dt+X^{2+\varepsilon}\int_{MX/H}^{U}\frac1{|t|^2}\left|M\left(\frac12+it\right)\right|^2\left|K\left(\frac12+it\right)\right|^2dt\\
&\ll I_3+I_4,
\end{align*}
where
$$
I_3=H^2X^\varepsilon\int_{X/H}^{MX/H}\left|M\left(\frac12+it\right)\right|^2\left|K\left(\frac12+it\right)\right|^2dt
$$
and
$$
I_4=X^{2+\varepsilon}\int_{MX/H}^{U}\frac1{|t|^2}\left|M\left(\frac12+it\right)\right|^2\left|K\left(\frac12+it\right)\right|^2dt.
$$
Similarly to the above argument to bound $I_1$, we get 
\begin{equation*}
I_3\ll H^2X^{1-\delta^2}+H^2X^{5\delta+\varepsilon}\sup_{|t|\leq MX/H}\left|K\left(\frac12+it\right)\right|^2.
\end{equation*}
Then by Lemma \ref{lem:subconvexity-partial-sums}, we have
\begin{equation}\label{I3}
I_3\ll H^2X^{1-\delta^2}+H^2X^{6\delta}\left((MX/H)^{2\theta}+H/M^2\right).
\end{equation}
Taking $U=X^{2026}$, by \eqref{bound for M(s)} and Lemma \ref{lem:mean-value-theorem}, 
\begin{align*}
    I_4&\ll  X^{2+\varepsilon}M^{1+\varepsilon}\sum_{n=1}^{[2026\log X]+1}\int_{2^{n}MX/H}^{2^{n+1}MX/H}\frac1{|t|^2}\left|K\left(\frac12+it\right)\right|^2dt\\
   &\ll X^{2+\varepsilon}M^{1+\varepsilon}\sup_{MX/H\leq T\leq X^{2026}}\frac1{T^2}\int_T^{2T}\left|K\left(\frac12+it\right)\right|^2dt\\
    &\ll X^{2+\varepsilon}M^{1+\varepsilon}\sup_{MX/H\leq T\leq X^{2026}}\frac{X/M+T}{T^2}\sum_{ k \sim X/M}\frac{|A(k)|^2}{k}\\
    &\ll X^{1+2\varepsilon}H^2M^{-2}+X^{1+2\varepsilon}H\ll H^2X^{1-\delta^2}.
\end{align*}
Thus, we have
$$
I_2\ll I_3+I_4\ll  H^2X^{1-\delta^2}+H^2X^{6\delta}\left((MX/H)^{2\theta}+H/M^2\right).
$$

All in all, \eqref{splitting} is bounded by
\begin{equation}
H^2X^{1-\delta^2}+H^2X^{6\delta}\left((MX/H)^{2\theta}+H/M\right)\ll  H^2X^{1-\delta^2}+X^{2\theta(1+\delta)+6\delta}H^{2-2\theta}.
\end{equation}

This is optimized for
\begin{equation}\label{choice of H}
H =
\begin{cases}
    X^{1+\delta -\frac{1}{2\theta} + \frac{6\delta + 2\delta^2}{2\theta}} & \text{if } \theta \geqslant 1/2, \\
    X^{6\delta} & \text{otherwise,}
\end{cases}
\end{equation}
which gives a final bound of 
\begin{equation}
H^2X^{1-2\delta^2}  =
\begin{cases}
    X^{3 +2\delta- \frac{1}{\theta} + \frac{6\delta + 2\delta^2}{\theta} - 2\delta^2} & \text{if } \theta \geqslant 1/2, \\ 
    X^{1+12\delta-2\delta^2} & \text{otherwise.}
\end{cases}
\end{equation}

From the second moment bound, we can derive pointwise bounds for most of the individual sums: 
\begin{coro}\label{coro1}
For all $C \geqslant 1$, we have that at least a proportion $1-C^{-2}$ of $x \sim X$ such that
\begin{equation}
\Sigma := \sum_{\substack{mk \in [x, x+H] \\ m \sim M \\ k \sim X/M \\ (m,k)=1}} A(mk) < CHX^{-\delta^2/2},
\end{equation}
for                                                                     $H=X^{1 +\delta-\frac{1}{2\theta} + \frac{6\delta + 2\delta^2}{2\theta}}$ if $\theta\geq 1/2$ and $H=X^{6\delta}$ if $0<\theta<1/2$.
\end{coro}

\begin{proof}
We first prove that the sum of coefficients is small for a large proportion of $x$'s, going from a result on average (the second moment bound of Proposition 1) to a pointwise statement. This is a Chebyshev type argument and only depends on the upper bound of the second moment, as in \cite{jaasaari_signs_2022}.

By the upper bound obtained in the first proposition, we have 
\begin{equation}
\int_2^{2X} |\Sigma|^2dx \ll  H^2X^{1-\delta^2},
\end{equation}
where $H=X^{1+\delta -\frac{1}{2\theta} + \frac{6\delta + 2\delta^2}{2\theta}}$ if $\theta\geq 1/2$ and $H=X^{6\delta}$ if $0<\theta<1/2$.
Cutting the integral between $x$'s such that the sum inside is larger (resp. smaller) than $B$ (say for subsets $X_1$ and $X_2$ respectively), we obtain that the contribution of $X_1$ is
\begin{equation}
H^2 X^{1-\delta^2} \gg \int_{X_1} |\Sigma|^2 \gg |X_1| B^2
\end{equation}
and therefore $|X_1| \leqslant H^2 X^{1-\delta^2} B^{-2}$. Taking $B$ so that this is $\asymp X$, i.e. $B = CHX^{-\delta^2/2}$,  we obtain that $|X_2| < (1-C^{-2})X$. In other words, we have $|\Sigma| < CHX^{-\delta^2/2}$ for at least $(1-C^{-2})X$ of the $x \sim X$, i.e. at least a proportion $1-C^{-2}$ of the $x\sim X$.
\end{proof}

\section{Lower bound}
\label{sec:lower-bound}

We prove a setting aiming at mimicking \cite[Lemma 22]{jaasaari_signs_2022}, which is valid for any $L$-function and provides some information on short intervals from the information on the full sums.
\begin{prop}
\label{prop:lower-bound}
For all $H=o(X)$, we have
\begin{equation}
\sum_{\substack{km \in [x, x+H]\\ m \sim M \\ (k,m)=1}} |A(km)| > HX^{\kappa  -1-2\varepsilon}
\end{equation}
for $\gg X^{2\kappa  - 1 -3\varepsilon}$ of the $x \sim X$.
\end{prop}

\begin{rk}
    Interestingly, this does not require anything on $\kappa $. It is a general claim that a lower bound on the whole sum implies lower bound on short intervals, provided we have some control on the Rankin-Selberg bounds.
\end{rk}

Let's separate the cases depending on the size of $A(m)$ to quantify the contribution of \textit{small} coefficients $A(m)$. Let $\alpha > 0$ be a parameter, and split the sum according to the size of $|A(m)|$ compared to $X^\alpha$.

By the assumption \eqref{lower-bound-coeff}, we have that
$$
\sum_{\substack{km \sim X \\ m \sim M \\ (k,m)=1}} |A(km)| =\sum_{m\sim M}|A(m)|\sum_{\substack{k\sim X/m \\  (k,m)=1}} |A(k)|\gg X^{\kappa-\varepsilon}.
$$
When $|A(m)| > X^\alpha$, we can use the Rankin-Selberg bound \eqref{R-S bound} and obtain
\begin{equation}
\sum_{\substack{km \sim X \\ m \sim M \\ (k,m)=1\\ |A(km)| > X^\alpha}} |A(km)| \leqslant X^{-\alpha} \sum_{\substack{km \sim X \\ m \sim M \\ (k,m)=1 }} |A(km)|^2 \ll X^{-\alpha}\sum_{m\sim M}|A(m)|^2\sum_{\substack{k\sim X/m \\  (k,m)=1}} |A(k)|^2\ll  X^{1-\alpha+\varepsilon}.
\end{equation}
Therefore,
\begin{equation}
\sum_{\substack{km \sim X \\ m \sim M \\ (k,m)=1\\ |A(km)| \leqslant X^\alpha}} |A(km)| = \sum_{\substack{km \sim X \\ m \sim M \\ (k,m)=1}} |A(km)| - \sum_{\substack{km \sim X \\ m \sim M \\ (k,m)=1\\ |A(km)| > X^\alpha}} |A(km)| \gg X^{\kappa -\varepsilon} - X^{1-\alpha+\varepsilon}
\end{equation}
and this is of size $X^{\kappa-\varepsilon} $ as long as $\alpha > 1-\kappa +2\varepsilon$. We assume this condition from now on, i.e. take $\alpha = 1-\kappa  + 2\varepsilon$. This implies in particular  
\begin{align}
\label{bound}
X^{\kappa -\varepsilon}\ll \sum_{\substack{km \sim X \\ m \sim M \\ (k,m)=1\\ |A(km)| \leqslant X^\alpha}} |A(km)|&\ll \sum_{n=0}^{[X/H]+1}\frac{1}{H}\int_{X+(n-1)H}^{X+nH}\sum_{\substack{km \in [X+nH, X+(n+1)H]  \\ m \sim M \\ (k,m)=1\\ |A(km)| \leqslant X^\alpha}}|A(km)|dx\nonumber\\
&\leq \sum_{n=0}^{[X/H]+1}\frac{1}{H}\int_{X+(n-1)H}^{X+nH}\sum_{\substack{km \in [x, x+2H]  \\ m \sim M \\ (k,m)=1\\ |A(km)| \leqslant X^\alpha}}|A(km)|dx \nonumber\\
&\leq \frac{1}{H}\int_{X-H}^{2X+H}\sum_{\substack{km \in [x, x+2H]  \\ m \sim M \\ (k,m)=1\\ |A(km)| \leqslant X^\alpha}}|A(km)|dx.
\end{align}
Here we have used the simple observation that for any $x\in[X+(n-1)H,X+nH]$, we have $[X+nH, X+(n+1)H]\subset [x,x+2H]$ and so
$$
\sum_{\substack{km \in [X+nH, X+(n+1)H]  \\ m \sim M \\ (k,m)=1\\ |A(km)| \leqslant X^\alpha}}|A(km)|\leq\sum_{\substack{km \in [x, x+2H]  \\ m \sim M \\ (k,m)=1\\ |A(km)| \leqslant X^\alpha}}|A(km)|.
$$
Consider the ``bad" set $\mathcal{S}$ of $x\sim X$ not satisfying the claimed bound (called $B$, threshold to be optimized later), i.e.
\begin{equation}
\mathcal{S} : =\Bigg\{ x \sim X \ : \ \sum_{\substack{km \in [x, x+2H]  \\ m \sim M \\ (k,m)=1\\ |A(km)| \leqslant X^\alpha}} |A(mk) |\leqslant B \Bigg\}.
\end{equation}
The points in $\mathcal{S}$ contribute in the above integral at most as  (bounding $|\mathcal{S}|$ by $X$)
\begin{equation}
\frac{1}{H}\int_{\mathcal{S}} \sum_{\substack{km \in [x, x+2H]  \\ m \sim M \\ (k,m)=1\\ |A(km)| \leqslant X^\alpha}} |A(km)| \ll \frac{1}{H} |\mathcal{S}| B \ll \frac{1}{H} X B.
\end{equation}
This does not contribute enough to ensure the size of the integral \eqref{bound} if $H^{-1} XB \ll X^{\kappa-\varepsilon} $, i.e. as soon as $B \ll HX^{\kappa  -1-\varepsilon}$. 
We therefore choose this frontier threshold $B = HX^{\kappa  -1 - 2\varepsilon}$. In this case, the size of the integral has to be determined by the points in the complementary $^c\mathcal{S}$ of $\mathcal{S}$, which is the set of $x \sim X$ where the claim is true. On ${}^c\mathcal{S}$, using the fact that the sum is over small enough coefficients $|A(mk)| \leqslant X^\alpha$, the contribution of the integral is at most
\begin{equation}
X^\kappa  \ll \frac{1}{H}\int_{{}^c \mathcal{S}} \sum_{\substack{km \in [x, x+H] \\ m \sim M \\ (k,m)=1 \\ |A(m)| \leqslant X^\alpha}} |A(km)| \ll \frac{1}{H}|{}^c\mathcal{S}| HX^\alpha.
\end{equation}
This implies that 
\begin{equation}
|{}^c\mathcal{S}| \gg X^{\kappa  - \alpha} = X^{2\kappa  - 1 - 3\varepsilon}
\end{equation}
 since we chose $\alpha = 1-\kappa  + 2\varepsilon$. We obtain the desired statement.  \qed

\section{Proof of the theorem}
\label{sec:conclusion}

We finish the proof of Theorem \ref{thm} in this section.

For the proportion $\gg X^{2\kappa  - 1 - \varepsilon}$ of $x \sim X$ arising from Corollary \ref{coro1} and Proposition \ref{prop:lower-bound},  we can write the consequence of the the upper and the lower bound we obtained:  
\begin{equation}
\left|\sum A(mk)\right| < CHX^{-\delta^2/2} < HX^{\kappa  -1} < \sum |A(mk)|
\end{equation}
where the central inequality is always valid for $X$ large enough provided 
$$\delta > \sqrt{2-2\kappa }.$$



 This detects a change of sign in the $A(mk)$'s in the interval $[x, x+H]$, and this happens for $\gg X^{2\kappa-1-\varepsilon}$ of $x$'s (since Corollary \ref{coro1} and Proposition \ref{prop:lower-bound} are simultaneously valid for such a number of $x$'s). Hence, there are at least $\gg X^{2\kappa - 1 -\varepsilon}/H$ sign changes in $[X, 2X]$. 
 
 If $\theta>1/2$, we can take (for the upper bound claim)  $H=X^{1+\delta -\frac{1}{2\theta} + \frac{6\delta + 2\delta^2}{2\theta}}$, and we therefore obtain that there are at least $X^{2\kappa - 2 + 1/2\theta-\delta- \frac{6\delta + 2\delta^2}{2\theta}}$ changes of signs in $A(m)$ for $m \in [0, X]$.
  Since we required $X^{6\delta} \ll H \ll X^{1-6\delta}$, we require ($X^{6\delta} \ll H$ is true when $\delta\leq1/6$, which is assured by the following formula)
  \begin{equation}\label{bound for delta}
      1 +\delta-\frac{1}{2\theta} + \frac{6\delta + 2\delta^2}{2\theta}\leq 1-6\delta\Longrightarrow\delta\leq \frac{1}{7\theta+3+\sqrt{(7\theta+3)^2+2}}.
  \end{equation}
  Hence,
  $$
  \sqrt{2-2\kappa}<\delta\leq\frac{1}{7\theta+3+\sqrt{(7\theta+3)^2+2}}\Longrightarrow\kappa\geq1-\frac{1}{2(7\theta+3+\sqrt{(7\theta+3)^2+2})^2}.
  $$
  
  If $\theta=1/2$, $\kappa\geq0.9971$ and $\kappa\rightarrow1$ as $\theta$ increasing.
If we take $\delta=\varepsilon$ as in Theorem \ref{thm}, then $\kappa\geq 1-\varepsilon$.

If $\theta\leq1/2$, we required (for the upper bound claim) that $H=X^{6\delta}$, and we obtain that there are at least $X^{2\kappa - 1 -6\delta}$ changes of signs in $A(m)$ for $m \in [0, X]$. The result is nontrivial when $\delta<(2\kappa-1)/6$. Hence
  $$
    \sqrt{2-2\kappa}<\delta<(2\kappa-1)/6\Longrightarrow \kappa\geq-\frac{17}{2}+3\sqrt{10}\geq 0.986833.
  $$
{The proof of Theorem \ref{thm} is finished by taking $\delta=\sqrt{2-2\kappa+\varepsilon}$. }\qed

 
 \subsection{An application to $\GSp(4)$}
 \label{sec:application}

We apply Theorem \ref{thm} to Siegel modular forms on $\GSp(4)$. From \cite[Theorem 4.2]{pitale_weissauer}, which is a consequence of the deep results of Weissauer \cite[Theorem~3.3]{weissauer} on the functoriality and Langlands conjectures for $\GSp(4)$, a Siegel modular Hecke eigenform which is not a Saito–Kurokawa lift satisﬁes the generalized Ramanujan conjecture. This can be shown to be equivalent to $|\mu(p^r)|\leq 36p^{r(k-3/2)}$ for forms in the orthogonal complement of the Maass space, where the $\mu(m)$ are the Hecke eigenvalues (see  \cite[Theorem 4.2]{pitale}). Using the relation $A(p^r)=\frac{\mu(p)}{p^{k-3/2}}$ (see \cite[equation  (7)]{pitale_sign_2008}) where the $A(m)$ are the coefficients of the spinor $L$-function, we deduce $|A(p)| \leqslant 36$.

This implies by the remark following the definition of our class of $L$-functions that $\vartheta=0$ is admissible, and therefore that $\kappa  = 1$ holds. Therefore, we can deduce from Theorem \ref{thm} the Corollary \ref{coro} on Siegel modular forms.




\subsection*{Acknowledgements.} This work was done when D. L. was visiting Shenzhen University, which we thank for excellent working environment and support. D. L. acknowledges the support of the CDP C2EMPI, together with the French State under the France-2030 program, the University of Lille, the Initiative of Excellence of the University of Lille and the European Metropolis of Lille for their funding and support of the R-CDP-24-004-C2EMPI project. Y. W. is supported by National Natural Science Foundation of China (Grant No. 12371006).

\bibliographystyle{abbrv}

\end{document}